\documentclass[11pt]{amsart}

\usepackage{amsmath,amssymb,mathtools}
\usepackage[T1]{fontenc}
\usepackage{lmodern,microtype}
\usepackage{booktabs,array}
\usepackage[hidelinks]{hyperref}
\allowdisplaybreaks[1]
\numberwithin{equation}{section}

\newtheorem{theorem}{Theorem}[section]
\newtheorem{proposition}[theorem]{Proposition}
\newtheorem{lemma}[theorem]{Lemma}
\newtheorem{corollary}[theorem]{Corollary}
\theoremstyle{remark}
\newtheorem{remark}[theorem]{Remark}

\newcommand{\Z}{\mathbb Z}
\newcommand{\R}{\mathbb R}
\newcommand{\C}{\mathbb C}
\newcommand{\E}{\mathcal E}
\newcommand{\HH}{\mathcal H}
\newcommand{\dd}{\,\mathrm d}

\newcommand{\Dzero}{\mathcal D_0}
\hypersetup{
 pdftitle={Sharp discrete Hardy constants in low dimensions},
 pdfauthor={Natanael Alpay},
 pdfkeywords={Discrete Hardy inequality, lattice Laplacian, sharp constant,
 reciprocal edge field, transition dimension}
}

\title[Sharp discrete Hardy constants in low dimensions]{Sharp discrete Hardy constants in low dimensions}
\author[N. Alpay]{Natanael Alpay}
\address{Department of Mathematics, University of California, Irvine, Irvine, CA 92697, USA}
\email{nalpay@uci.edu}
\subjclass[2020]{39A12, 26D15, 47B39}
\keywords{Discrete Hardy inequality, lattice Laplacian, sharp constant,
reciprocal edge field, transition dimension}
\date{}
\begin{document}

\begin{abstract}
For $d\ge3$, let $\lambda_d$ denote the sharp constant in the
nearest-neighbor Hardy inequality on $\mathbb Z^d$ with Euclidean
inverse-square weight for functions vanishing at the origin.
Recent work established that the continuous Hardy coefficient
\[
C_d=\frac{(d-2)^2}{4}
\]
is sharp in dimensions three and four, while it is not sharp
for $d\ge9$. We prove
$\lambda_d=C_d$ for the remaining cases $5\le d\le8$.
The proof uses reciprocal edge fields and pointwise estimates
for the associated vertex weights. The exterior estimates follow
from polynomial inequalities, and the weights near the origin
are treated by explicit estimates and local corrections. 
\end{abstract}

\maketitle

\section{Introduction}

This paper continues the study started in \cite{Alpay2026} of the sharp
nearest-neighbor Hardy inequality on $\Z^3$.
For $d\ge3$, let
$e_1,\ldots,e_d$ denote the standard basis of $\Z^d$.  For a finitely supported
function $u:\Z^d\to\C$, define
\begin{align}
 \E_d(u)
 &=
 \sum_{x\in\Z^d}\sum_{j=1}^d
 |u(x+e_j)-u(x)|^2,
 \label{eq:energy-d}\\
 \HH_d(u)
 &=
 \sum_{x\in\Z^d\setminus\{0\}}
 \frac{|u(x)|^2}{|x|^2}.
 \label{eq:mass-d}
\end{align}
Every undirected nearest-neighbor edge is counted once in
\eqref{eq:energy-d}, including the edges incident with the origin.  Set
\[
 \Dzero(\Z^d)
 =
 \{u\in C_c(\Z^d;\C):u(0)=0\},
\]
and define
\begin{equation}\label{eq:lambda-d}
 \lambda_d
 =
 \inf_{\substack{u\in\Dzero(\Z^d)\\u\ne0}}
 \frac{\E_d(u)}{\HH_d(u)},
 \qquad
 C_d=\frac{(d-2)^2}{4}.
\end{equation}
Thus $\lambda_d$ is the largest coefficient for which
\[
 \E_d(u)\ge\lambda_d\HH_d(u)
\]
holds for every $u\in\Dzero(\Z^d)$.  The number $C_d$ is the sharp
coefficient in the corresponding continuous Hardy inequality on $\R^d$.
A continuum approximation gives
\begin{equation}\label{eq:universal-upper}
 \lambda_d\le C_d,\qquad d\ge3,
\end{equation}
so the fixed-dimensional problem is to decide when the lattice retains the
continuous sharp coefficient.

The first exact low-dimensional value was found in
\cite{Alpay2026}, where we proved
\[
 \lambda_3=\frac14
\]
by means of an explicit reciprocal edge field, an edgewise completion of
squares, and a concavity argument for the associated vertex weight.
Independently, Lizama \cite{Lizama2026} proved $\lambda_3=\frac14$ and
\[
 \lambda_4=1,
\]
using shifted radial supersolutions, angular convexity, and a discrete
ground-state representation.  Moreover, the same work proved
\begin{equation}\label{eq:lizama-high-d}
 \lambda_d<C_d,\qquad d\ge9,
\end{equation}
and showed that its shifted radial-power family cannot establish $C_d$ from
dimension five onward.  
Therefore the transition transition
window for dimensions 
$
 5\le d\le8
$
is left open to verify if indeed $\lambda_d=C_d$ still holds.\\

For large dimensions, Gupta
\cite{Gupta2023} proved that $\lambda_d$ has order $d$, in contrast with the
quadratic growth of $C_d$.  Gupta \cite{Gupta2026} and, independently, Huang
and Ye \cite{HuangYe2026} subsequently proved
\[
 \lim_{d\to\infty}\frac{\lambda_d}{d}=2.
\]
Lizama \cite{LizamaAsymptotics2026} refined this to
\[
 \lambda_d
 =
 2d-4-\frac{20}{3d}+O(d^{-2}),
 \qquad d\to\infty.
\]
Thus the equality $\lambda_d=C_d$ is necessarily a low-dimensional
phenomenon, what remained open is whatever this really is the case for all dimensions $5\leq d\leq 8$. \\

The main result of this paper closes that gap.  We prove
\begin{equation}\label{eq:intro-main}
 \lambda_d=C_d=\frac{(d-2)^2}{4},
 \qquad 5\le d\le8.
\end{equation}
This resolves the remaining cases of the conjecture stated
in \cite[Section~6]{Lizama2026}.

In particular,
\[
 \lambda_5=\frac94,\qquad
 \lambda_6=4,\qquad
 \lambda_7=\frac{25}{4},\qquad
 \lambda_8=9.
\]
Combining \eqref{eq:intro-main} with the known values in dimensions three and
four and with \eqref{eq:lizama-high-d}, we obtain the exact relation,
\begin{equation}\label{eq:intro-transition}
 \lambda_d=C_d
 \quad\Longleftrightarrow\quad
 3\le d\le8,
 \qquad d\ge3.
\end{equation}
Thus dimension nine is the first dimension in which the continuum Hardy
coefficient ceases to be sharp for the nearest-neighbor lattice problem.\\

The proof develops the reciprocal-edge method of \cite{Alpay2026}.
For $d=5,6,7$, the main analytic ingredient is a quadratic rational
reciprocal field whose first-order behavior agrees with the formal
Hardy ground-state ratios. Pairing the neighbors $x+e_j$ and $x-e_j$
expresses the induced vertex weight as a sum of functions of $x_j^2$.
The concavity argument used in dimension three no longer persists
in the required form; instead we reduce the affine lower bound to
an explicit quadratic polynomial inequality. We prove this inequality
by completing a square, with the remaining terms nonnegative on the
required range. This yields the coefficient $C_d$ outside a finite
set for $d=5,6,7$. The remaining points are treated by explicit
estimates and local corrections.\\

Dimension eight is the last dimension in which the lattice retains
the continuum sharp constant and requires a different rational
reciprocal field. We prove that its vertex weight is greater than
$9/|x|^2$ for $|x|^2\ge13$. The exterior estimate reduces to a
quartic polynomial inequality. A quadratic lower bound for this
quartic reduces the proof to elementary polynomial estimates.
We then modify the field inside
\[
 S_{12}=\{x\in\mathbb Z^8:1\le |x|^2\le12\},
\]
leaving all crossing and exterior edges unchanged. Coordinate
permutations and sign changes partition this set into $26$ symmetry
classes. A positive integer-valued function, constant on each class,
determines the interior modification. Its values and all required
finite inequalities are given in Section~\ref{sec:d8}. Together with
the exterior estimate, this gives $\lambda_8\ge9$; the continuum
upper bound then yields $\lambda_8=9$.\\

The paper is organized as follows. Section~\ref{sec:main-results}
states the main theorem and introduces the quadratic reciprocal
field. Section~\ref{sec:edge-proof} presents the reciprocal edge
decompositions and the finite-gluing lemma.
Section~\ref{sec:exterior} proves the exterior estimate in dimensions
five through seven. Section~\ref{sec:local} treats the weights near
the origin in dimensions four through seven.
Section~\ref{sec:d8} gives the dimension-eight construction.
Section~\ref{sec:sharpness} proves the continuum upper bound, and
Section~\ref{sec:completion} completes the proof of the transition
theorem. Appendices~\ref{app:polynomials} and
\ref{app:d8-polynomials} give the polynomial identities used in the
two exterior estimates.

\section{Main results}\label{sec:main-results}

\begin{theorem}[Sharp constants]\label{thm:main}
For every $d\in\{5,6,7,8\}$,
\begin{equation}\label{eq:sharp-constants}
 \lambda_d=\frac{(d-2)^2}{4}.
\end{equation}
Equivalently, for every $u\in\Dzero(\Z^d)$,
\begin{equation}\label{eq:main-hardy}
 \E_d(u)\ge\frac{(d-2)^2}{4}\HH_d(u).
\end{equation}
For $d=5,6$, \eqref{eq:main-hardy} holds for every finitely supported
$u:\Z^d\to\C$, without imposing $u(0)=0$.  In every stated admissible class
the coefficient is sharp, and equality is not attained by a nonzero finitely
supported function.
\end{theorem}

Thus
\[
 \lambda_5=\frac94,\qquad
 \lambda_6=4,\qquad
 \lambda_7=\frac{25}{4},\qquad
 \lambda_8=9.
\]
For $d=5,6,7$, the proof uses the quadratic reciprocal field below.
Dimension eight uses a different rational reciprocal field and an
explicit local correction, given in Section~\ref{sec:d8}.

\begin{corollary}\label{cor:transition}
For every integer $d\ge3$,
\[
 \lambda_d=\frac{(d-2)^2}{4}=C_d
 \quad\Longleftrightarrow\quad
 3\le d\le8.
\]
In particular, dimension nine is the first dimension in which the continuous
Hardy coefficient fails to be sharp on $\Z^d$.
\end{corollary}

\begin{proof}
The equality in dimensions $5\le d\le8$ is Theorem~\ref{thm:main}.
The values $\lambda_3=1/4$ and $\lambda_4=1$ are known from
\cite{Alpay2026,Lizama2026}, whereas \cite{Lizama2026} proves
$\lambda_d<C_d$ for every $d\ge9$.
\end{proof}

\begin{remark}\label{rem:d4}
The same quadratic reciprocal-edge construction also recovers $\lambda_4=1$.
This value was proved previously in \cite{Lizama2026} by a different method.
We include a remark about the idea of the proof for $d=4$ but omit the details as it was done previously. 
\end{remark}

\subsection{Main steps of the proof}

For $4\le d\le7$, write $x\sim y$ when $x$ and $y$ are nearest neighbors.
For an ordered neighboring pair, put
\[
 n=|x|^2,\qquad m=|y|^2,\qquad S=n+m,\qquad\delta=m-n,
\]
and set
\begin{equation}\label{eq:parameters}
 r=\frac{d-2}{4},\qquad b=\frac{(r+1)(r+2)}3.
\end{equation}
The construction below extends the reciprocal-edge method of \cite[Appendix~A]{Alpay2026}, where the ground-state motivation and the role of reciprocity are discussed. Here we introduce a quadratic rational field with the parameters specified above. The required pointwise lower bound for its associated vertex weight is proved in Section~\ref{sec:exterior}.\\

Define the base field by
\begin{equation}\label{eq:base-field}
 a^{0,d}_{xy}=a_d(n,m)
 =\frac{S^2+b\delta^2-rS\delta}
        {S^2+b\delta^2+rS\delta}.
\end{equation}
Its numerator and denominator are positive because
\begin{equation}\label{eq:positive-quadratic}
 S^2+b\delta^2\pm rS\delta
 =\left(S\pm\frac r2\delta\right)^2
  +\frac{r^2+12r+8}{12}\delta^2>0.
\end{equation}
On a lattice edge $S>0$, so the two terms on the right cannot vanish
simultaneously.  Interchanging $x$ and $y$ reverses $\delta$, and therefore
\[
 a^{0,d}_{yx}=(a^{0,d}_{xy})^{-1}.
\]

Define the corresponding base weight by
\begin{equation}\label{eq:base-weight}
 W_d^0(x)=\sum_{y\sim x}(1-a^{0,d}_{xy}).
\end{equation}
Individual summands may be negative.  After pairing opposite lattice
directions, however, their sum admits an affine lower bound.

\begin{proposition}[Exterior weight estimate]\label{prop:exterior}
Let $5\le d\le7$, and let
\begin{equation}\label{eq:thresholds}
 (L_5,L_6,L_7)=(2,3,4).
\end{equation}
Then
\begin{equation}\label{eq:exterior-weight}
 W_d^0(x)>\frac{C_d}{|x|^2}
 \qquad\text{whenever }|x|^2\ge L_d.
\end{equation}
\end{proposition}

The same calculation with $d=4$ gives
\[
 W_4^0(x)>\frac1{|x|^2},
 \qquad |x|^2\ge2.
\]
The dimension-four verification is given in
Remark~\ref{rem:d4-check}.\\

Fix $x\ne0$ and put $n=|x|^2$. To combine the contributions from
opposite neighbors, for $0\le s\le n.$ Define
\[
 f_{d,n}(s)
 =
 2-a_d(n,n+1+2\sqrt{s})
  -a_d(n,n+1-2\sqrt{s}).
\]
Since $|x\pm e_j|^2=n+1\pm2x_j$, the value $f_{d,n}(x_j^2)$
is precisely the combined contribution of $x+e_j$ and $x-e_j$
to the vertex weight. If $x_j<0$, the two arguments are simply
interchanged. Summing over coordinate directions gives
\[
 W_d^0(x)=\sum_{j=1}^d f_{d,n}(x_j^2).
\]
For $5\le d\le7$ and $n\ge L_d$, the key estimate is
\begin{equation}\label{eq:main-minorant}
 f_{d,n}(s)\ge f_{d,n}(0)+\frac{s}{n}
       \left(\frac{C_d}{n}-d f_{d,n}(0)\right),
 \qquad 0\le s\le n.
\end{equation}
Since $\sum_jx_j^2=n$, summing \eqref{eq:main-minorant} gives
$W_d^0(x)\ge C_d/n$.  The difference between the two sides of
\eqref{eq:main-minorant} is a nonnegative factor, positive for $s>0$, times a quadratic
polynomial in $s$. We prove its positivity by
completing a square. The coefficient of the remaining linear term
and the remaining constant term are polynomials in $n-L_d$ with
strictly positive coefficients. Section~\ref{sec:exterior} and
Appendix~\ref{app:polynomials} give the calculation.

In dimension eight, we replace \eqref{eq:base-field} by the rational
field \eqref{eq:d8-rational-field}. Proposition~\ref{prop:d8-exterior}
gives the required strict lower bound for $|x|^2\ge13$. The correction
on $S_{12}$ is given by the $26$ values in Table~\ref{tab:d8-local}.
The exterior estimate and the local inequalities together give
Proposition~\ref{prop:d8-lower}.

Sharpness follows from the continuum upper bound below.

\begin{proposition}[Continuum upper bound]\label{prop:upper}
For every integer $d\ge3$, there is a sequence of nonzero
$u_k\in\Dzero(\Z^d)$ such that
\[
 \frac{\E_d(u_k)}{\HH_d(u_k)}\longrightarrow C_d.
\]
In particular, $\lambda_d\le C_d$.
\end{proposition}

\begin{remark}[Sharp constants and optimal weights]
Our result concerns the best coefficient multiplying the fixed weight
$|x|^{-2}$.  We do not claim that the weights constructed below are critical
or pointwise optimal Hardy weights.
\end{remark}

\section{Reciprocal edge decompositions}\label{sec:edge-proof}

Let
\[
 \mathcal E_+^{(d)}=\{(x,x+e_j):x\in\Z^d,\ 1\le j\le d\}.
\]
Every undirected nearest-neighbor edge appears exactly once in this set.
Since $a_{yx}=a_{xy}^{-1}$, reversing an edge does not change the
squared term in the decomposition below.

\begin{lemma}[Two-point identity]\label{lem:edge-square}
For $a>0$ and $z,w\in\C$,
\begin{equation}\label{eq:two-point-square}
 |z-w|^2=(1-a)|z|^2+(1-a^{-1})|w|^2
       +\left|a^{1/2}z-a^{-1/2}w\right|^2.
\end{equation}
\end{lemma}
\begin{proof}
The last square is
$a|z|^2+a^{-1}|w|^2-2\operatorname{Re}(z\overline w)$.
Adding the other two terms leaves
$|z|^2+|w|^2-2\operatorname{Re}(z\overline w)=|z-w|^2$.
\end{proof}

\begin{proposition}[Full-lattice decomposition]\label{prop:edge}
Let $a_{xy}>0$ be reciprocal on every nearest-neighbor edge of $\Z^d$,
and define
\[
 W(x)=2d-\sum_{y\sim x}a_{xy}.
\]
For every finitely supported $u:\Z^d\to\C$,
\begin{equation}\label{eq:edge}
 \E_d(u)-\sum_xW(x)|u(x)|^2
 =\sum_{(x,y)\in\mathcal E_+^{(d)}}
   \left|\sqrt{a_{xy}}u(x)-a_{xy}^{-1/2}u(y)\right|^2.
\end{equation}
\end{proposition}
\begin{proof}
Apply \eqref{eq:two-point-square} to each forward edge.
At a fixed vertex $v$, the edges $(v,v+e_j)$ contribute
$\sum_j(1-a_{v,v+e_j})|u(v)|^2$. The edges $(v-e_j,v)$ contribute
$\sum_j(1-a_{v,v-e_j})|u(v)|^2$, where reciprocity has been used
for the second endpoint. The total coefficient is therefore
\[
 \sum_{j=1}^d\bigl[(1-a_{v,v+e_j})+(1-a_{v,v-e_j})\bigr]=W(v).
\]
Only edges meeting the finite support of $u$ contribute, so all sums
are finite and the regrouping is legitimate. The remaining terms are
the edge squares in \eqref{eq:edge}.
\end{proof}

Set $X_d=\Z^d\setminus\{0\}$. We now give the corresponding identity
for functions that vanish at the origin.
\begin{proposition}[Dirichlet decomposition]\label{prop:punctured}
Assign positive reciprocal factors to edges with both endpoints in $X_d$.
Set
\begin{equation}\label{eq:punctured-weight}
 V(x)=2d-\sum_{\substack{y\sim x\\y\ne0}}a_{xy},\qquad x\in X_d.
\end{equation}
For every $u\in\Dzero(\Z^d)$,
\begin{equation}\label{eq:punctured}
 \E_d(u)-\sum_{x\ne0}V(x)|u(x)|^2
 =\sum_{\substack{(x,y)\in\mathcal E_+^{(d)}\\x,y\ne0}}
 \left|\sqrt{a_{xy}}u(x)-a_{xy}^{-1/2}u(y)\right|^2.
\end{equation}
\end{proposition}
\begin{proof}
Apply the two-point identity only to edges with nonzero endpoints.
If $x\sim0$, the remaining origin edge contributes $|u(x)|^2$,
since $u(0)=0$. Thus its diagonal contribution is retained in $2d$
but no factor for that edge is subtracted in \eqref{eq:punctured-weight}.
Regrouping the other terms as in Proposition~\ref{prop:edge} proves
\eqref{eq:punctured}.
\end{proof}

\begin{remark}[No gradient compatibility is required]
If $a_{xy}=H(y)/H(x)$ for a positive vertex function $H$, the factors
multiply to one around every closed path. The two propositions require positivity and reciprocity, but not this
closed-path condition. We use this fact when modifying the field in
Section~\ref{sec:local}.
\end{remark}

\begin{lemma}[Finite gluing]\label{lem:gluing}
Let $S\subset X_d$ be finite, and let $a^0$ be a positive reciprocal
field on edges with nonzero endpoints. Suppose its punctured weight
satisfies $V^0(x)\ge\kappa/|x|^2$ for $x\in X_d\setminus S$.
For $x\in S$, define
\begin{equation}\label{eq:core-diagonal}
 D_x=2d-\frac{\kappa}{|x|^2}
     -\sum_{\substack{y\sim x\\y\in X_d\setminus S}}a^0_{xy},
\end{equation}
and
\begin{equation}\label{eq:core-operator}
 (Mh)(x)=D_xh(x)-\sum_{\substack{y\sim x\\y\in S}}h(y).
\end{equation}
If $h:S\to(0,\infty)$ satisfies $Mh\ge0$, there is a positive
reciprocal field with punctured weight at least $\kappa/|x|^2$
at every $x\ne0$.
\end{lemma}

\begin{proof}
Set $a_{xy}=h(y)/h(x)$ when both endpoints lie in $S$, and retain
$a^0_{xy}$ on every other edge. In particular every edge crossing the
boundary of $S$ remains unchanged. Each exterior vertex therefore
retains all its incident factors. For $x\in S$, its new weight satisfies
\[
 V(x)-\frac{\kappa}{|x|^2}
 =D_x-\sum_{\substack{y\sim x\\y\in S}}\frac{h(y)}{h(x)}
 =\frac{(Mh)(x)}{h(x)}\ge0.
\]
The field is positive and reciprocal by construction.
\end{proof}

\section{Proof of the exterior weight estimate}\label{sec:exterior}

Expanding \eqref{eq:base-field} gives
\begin{equation}\label{eq:r-polynomial}
 a_d(n,m)=
 \frac{(r^2+6r+5)n^2-2(r^2+3r-1)nm+(r^2+5)m^2}
 {(r^2+5)n^2-2(r^2+3r-1)nm+(r^2+6r+5)m^2}.
\end{equation}
Table~\ref{tab:field} lists the coefficients for each dimension after
multiplication by positive common factors to clear denominators.

\begin{table}[ht]
\centering
\caption{Coefficients in $a_d(n,m)=(\alpha_dn^2+\beta_dnm+\gamma_dm^2)/(\gamma_dn^2+\beta_dnm+\alpha_dm^2)$ for $4\le d\le7$.}
\label{tab:field}
\begin{tabular}{crrr}
\toprule
$d$&$\alpha_d$&$\beta_d$&$\gamma_d$\\
\midrule
4&11&$-2$&7\\
5&161&$-58$&89\\
6&2&$-1$&1\\
7&75&$-46$&35\\
\bottomrule
\end{tabular}
\end{table}

Fix $n\ge1$. For $0\le s\le n$, define
\begin{equation}\label{eq:f-definition}
 f_{d,n}(s)
 =2-a_d(n,n+1+2\sqrt{s})-a_d(n,n+1-2\sqrt{s}).
\end{equation}
The rational formula for $a_d(n,m)$ is well defined for real
$n>0$ and $m\ge0$, since its denominator is positive by
\eqref{eq:positive-quadratic}. For $0\le s\le n$, we have
\[
 n+1-2\sqrt{s}\ge n+1-2\sqrt{n}
 =(\sqrt{n}-1)^2\ge0.
\]
Thus both evaluations of $a_d$ in \eqref{eq:f-definition}
are well defined.

At a lattice point $x\ne0$ with $n=|x|^2$, the identities
$|x\pm e_j|^2=n+1\pm2x_j$ show that $f_{d,n}(x_j^2)$
combines the contributions from the two opposite neighbors
$x+e_j$ and $x-e_j$. Consequently,
\begin{equation}\label{eq:W-paired}
 W_d^0(x)=\sum_{j=1}^d f_{d,n}(x_j^2).
\end{equation}
If $x_j<0$, the two neighboring squared radii are interchanged, which
does not change their sum in \eqref{eq:f-definition}.

Put
\begin{equation}\label{eq:pqE}
 p=r+2,\qquad q=\frac{(r+1)(r+5)}3,\qquad
 E=4n^2+2pn+q,
\end{equation}
and write
\begin{align}
 D&=(E+4qs)^2-16(pn+q)^2s,\label{eq:D}\\
 N&=(2n+1+4s)(E+4qs)-16(n+1)(pn+q)s.\label{eq:N}
\end{align}
To derive the paired expression, note that
\[
 1-a_d(n,m)=\frac{2r(n+m)(m-n)}
 {(n+m)^2+b(m-n)^2+r(n+m)(m-n)}.
\]
For $m=n+1\pm2\sqrt s$, the two denominators are
\[
 E+4qs\pm4(pn+q)\sqrt s.
\]
Both are positive by \eqref{eq:positive-quadratic}, so their product
$D$ is positive. The numerator factors $(n+m)(m-n)$ are
\[
 2n+1+4s\pm4(n+1)\sqrt s.
\]
Adding the two fractions now gives
\begin{equation}\label{eq:f-general}
 f_{d,n}(s)=\frac{4rN}{D},
 \qquad f_{d,n}(0)=\frac{4r(2n+1)}E.
\end{equation}

\begin{lemma}\label{lem:minorant}
Let $n>0$, let $d\ge1$ be an integer, and let
$\alpha,\beta,\kappa\in\R$.
Suppose $f:[0,n]\to\R$ satisfies
\[
 f(s)\ge\alpha+\beta s
 \quad\text{for every }s\in[0,n],
 \qquad
 d\alpha+n\beta\ge\frac{\kappa}{n}.
\]
Then
\[
 \sum_{j=1}^d f(s_j)\ge\frac{\kappa}{n}
 \quad\text{whenever }s_j\ge0,\quad
 \sum_{j=1}^d s_j=n.
\]
In particular, it is sufficient that
\begin{equation}\label{eq:minorant}
 f(s)\ge f(0)+\frac{s}{n}
 \left(\frac{\kappa}{n}-df(0)\right),
 \qquad 0\le s\le n.
\end{equation}
If \eqref{eq:minorant} is strict for $s>0$, the resulting
summed inequality is strict.
\end{lemma}

\begin{proof}
Since the $s_j$ are nonnegative and sum to $n$, each belongs
to $[0,n]$. Applying the assumed bound
$f(s_j)\ge\alpha+\beta s_j$ for each $j$ and summing gives
\[
 \sum_{j=1}^d f(s_j)
 \ge d\alpha+\beta\sum_{j=1}^d s_j
 =d\alpha+n\beta
 \ge\frac{\kappa}{n}.
\]
For \eqref{eq:minorant}, take
\[
 \alpha=f(0),
 \qquad
 \beta=\frac1n\left(\frac{\kappa}{n}-df(0)\right),
\]
so that $d\alpha+n\beta=\kappa/n$.
Finally, at least one $s_j$ is positive because their sum
is $n>0$. If \eqref{eq:minorant} is strict for positive
arguments, at least one summand inequality is strict,
and hence so is their sum.
\end{proof}

If $f$ is concave, its endpoint chord gives
\[
 f(s)\ge f(0)+\frac{s}{n}\bigl(f(n)-f(0)\bigr),
\]
and hence $\sum_jf(s_j)\ge f(n)+(d-1)f(0)$.
This is the concavity argument used in dimension three \cite{Alpay2026}. For the quadratic
field, we prove \eqref{eq:minorant} directly. The paired function need
not be concave: in dimension six,
$f_{6,3}''(1)=4096/65219>0$.\\

Keep the notation in \eqref{eq:pqE}, \eqref{eq:D}, and \eqref{eq:N}. Define
\begin{align}
 H&=(4r+2)n(2n+1)-rE\notag\\
  &=(r+1)\left(4n^2-2(r-1)n-\frac{r(r+5)}3\right),\label{eq:H}\\
 J&=pn+q.\notag
\end{align}
Set
\begin{align}
 \mathsf A&=16q^2H,\notag\\
 \mathsf B&=H(8qE-16J^2)+16n^2q\bigl[E-q(2n+1)\bigr],\notag\\
 \mathsf G&=HE^2+n^2\bigl[4E^2-4q(2n+1)E-16(n+1)JE+16(2n+1)J^2\bigr].\label{eq:general-coefficients}
\end{align}
Since $d=4r+2$ and $C_d=4r^2$, subtraction of the affine expression
in \eqref{eq:minorant} yields
\begin{align}
 &f_{d,n}(s)-f_{d,n}(0)
 -\frac{s}{n}\left(\frac{4r^2}{n}-(4r+2)f_{d,n}(0)\right)\notag\\
 &\qquad=\frac{4rs}{n^2ED}
       \bigl[\mathsf A(n)s^2+\mathsf B(n)s+\mathsf G(n)\bigr].
 \label{eq:gap-general}
\end{align}
Multiplying the left-hand side by $n^2ED$ gives
\begin{align*}
 4r n^2EN-4r n^2D(2n+1)
 -sD\bigl[4r^2E-4r(4r+2)n(2n+1)\bigr].
\end{align*}
Collecting powers of $s$ gives
$4rs(\mathsf A s^2+\mathsf B s+\mathsf G)$, with the coefficients in
\eqref{eq:general-coefficients}. It remains to prove that this quadratic
polynomial is positive.

\begin{lemma}[Positivity in dimensions five through seven]
\label{lem:quadratic-positivity}
For $d\in\{5,6,7\}$, take $r=(d-2)/4$ in
\eqref{eq:general-coefficients}. Then
\[
 \mathsf A(n)s^2+\mathsf B(n)s+\mathsf G(n)>0
 \qquad\text{for all real }n\ge L_d,\quad s\ge0,
\]
with the thresholds in \eqref{eq:thresholds}.
\end{lemma}
\begin{proof}
Appendix~\ref{app:polynomials} gives positive constants
$\theta_d,\eta_d$, constants $c_d$, and polynomials $Q_d,T_d$
with strictly positive coefficients such that, for $t=n-L_d$,
\begin{equation}\label{eq:positive-square}
 \frac{\mathsf A s^2+\mathsf B s+\mathsf G}{\theta_d}
 =\frac{\mathsf A}{\theta_d}
       \left(s-\frac n4-c_d\right)^2
   +\eta_d Q_d(t)s+T_d(t).
\end{equation}
Since $3/4\le r\le3/2$ and $n\ge L_d\ge2$,
\[
 H=(r+1)\left(4n^2-2(r-1)n-\frac{r(r+5)}3\right)
 \ge(r+1)\left(4n^2-n-\frac{13}{4}\right)>0.
\]
Here
$4n^2-n-13/4=4(n-2)^2+15(n-2)+43/4$.
Since $q>0$, we obtain $\mathsf A=16q^2H>0$.
For $t,s\ge0$, the first two terms on the right of
\eqref{eq:positive-square} are nonnegative, and $T_d(t)>0$.
Multiplying by $\theta_d>0$ proves the result.
\end{proof}

\begin{proof}[Proof of Proposition~\ref{prop:exterior}]
By Lemma~\ref{lem:quadratic-positivity}, the right-hand side of
\eqref{eq:gap-general} is positive for $s>0$ and zero for $s=0$,
whenever $n\ge L_d$. Its denominators are positive by the preceding
pairing calculation. Apply Lemma~\ref{lem:minorant} to
\eqref{eq:W-paired}, with $\kappa=C_d$ and $s_j=x_j^2$.
Because $n>0$, at least one $s_j$ is positive. The strict estimate
\eqref{eq:exterior-weight} follows.
\end{proof}

\section{The weight near the origin}\label{sec:local}

The edge factors used below are unchanged under coordinate
permutations and independent sign changes of the coordinates,
applied simultaneously to both endpoints of each edge.
These transformations also preserve nearest-neighbor adjacency.
Consequently, vertices related by such transformations have the
same weight, and it suffices to compute the weight at one
representative of each symmetry class.
These are pointwise bounds; no symmetry assumption is imposed
on the test function $u$.

\subsection{Dimensions four and five}

For $d=4,5$, the base field gives the required estimate
outside the first shell. The dimension-four exterior
estimate is verified in Remark~\ref{rem:d4-check}.
At the origin,
\[
 W_d^0(0)=\frac{4dr}{q}>0.
\]

A first-shell vertex is a signed permutation of $e_1$. Its neighbors have
squared radii $0$ once, $4$ once, and $2$ with multiplicity $2(d-1)$, so
\begin{equation}\label{eq:first-shell-formula}
 W_d^0(e_1)-C_d
 =2d-a_d(1,0)-a_d(1,4)-2(d-1)a_d(1,2)-C_d.
\end{equation}

For $d=4$, direct substitution gives
\[
 W_4^0(0)=\frac{32}{11},
 \qquad
 W_4^0(e_1)-1=\frac{499}{1645}>0.
\]
Together with the exterior estimate indicated above, this recovers the
known value $\lambda_4=1$ by the present reciprocal-edge method.

For $d=5$,
\[
 a_5(1,0)=\frac{161}{89},\qquad
 a_5(1,4)=\frac{451}{811},\qquad
 a_5(1,2)=\frac{401}{617},
\]
and therefore
\[
 W_5^0(0)=\frac{720}{161}>0,
 \qquad
 W_5^0(e_1)-\frac94
 =\frac{33054525}{178137772}>0.
\]
Combining this with Proposition~\ref{prop:exterior} proves that the base
weight is strictly above $9/(4|x|^2)$ at every nonzero vertex of $\Z^5$.

\subsection{Dimension six}
Start with the base field on the full lattice. Change only the following
ordered-edge values, assigning reciprocal values on reversal:
\begin{equation}\label{eq:d6-modification}
 a^{(6)}_{0,e}=\frac23\quad(|e|=1),\qquad
 a^{(6)}_{xy}=\frac7{12}
 \quad(|x|^2=1,\ |y|^2=2,\ x\sim y).
\end{equation}
The other factors remain $a^{0,6}_{xy}$. Since both endpoints of each
changed edge lie in $\{|x|^2\le2\}$, the weight is unchanged when
$|x|^2\ge3$. Proposition~\ref{prop:exterior} therefore still applies
at those vertices.

Let $W_6(x)=12-\sum_{y\sim x}a^{(6)}_{xy}$.
At the origin,
\begin{equation}\label{eq:d6-origin}
 W_6(0)=12\left(1-\frac23\right)=4.
\end{equation}
If $|x|^2=1$, the neighboring squared radii are $0$ once, $4$ once,
and $2$ ten times. Since $a_6(1,4)=14/29$,
\begin{equation}\label{eq:d6-shell1}
 W_6(x)=12-\frac32-\frac{14}{29}-10\frac7{12}
       =4+\frac{16}{87}>4.
\end{equation}
If $|x|^2=2$, the neighboring squared radii are $1$ twice, $5$ twice,
and $3$ eight times. The unchanged factors are
$a_6(2,5)=23/44$ and $a_6(2,3)=11/16$, so
\begin{equation}\label{eq:d6-shell2}
 W_6(x)=12-2\frac{12}{7}-2\frac{23}{44}-8\frac{11}{16}
       =2+\frac2{77}>2.
\end{equation}
Every vertex with squared norm one or two is a signed permutation of
$e_1$ or $e_1+e_2$, respectively. Thus the two calculations cover all remaining nonzero vertices.

\subsection{Dimension seven}
Work on $X_7=\Z^7\setminus\{0\}$. Change only
\begin{equation}\label{eq:d7-modification}
 a^{(7)}_{xy}=\frac47\quad(n=1,m=2),\qquad
 a^{(7)}_{xy}=\frac{16}{25}\quad(n=2,m=3),
\end{equation}
with reciprocal values on reversal. All other factors on edges with
nonzero endpoints remain unchanged. At a vertex of squared norm at least
four, no incident edge has been changed and no neighbor is the origin.
Its punctured weight therefore equals the base weight, so
Proposition~\ref{prop:exterior} applies.

For $n=1,2,3$, a vertex of squared norm $n$ has exactly $n$ coordinates
equal to $\pm1$, the rest being zero. Changing a nonzero coordinate
produces a neighbor of squared radius $n-1$ or $n+3$; changing a zero
coordinate produces two neighbors of squared radius $n+1$.
The needed base ratios are
\[
 a_7(1,4)=\frac{451}{1051},\quad
 a_7(2,5)=\frac{143}{311},\quad
 a_7(3,6)=\frac{41}{81},\quad
 a_7(3,4)=\frac{683}{963}.
\]
Writing $V_{7,n}$ for the punctured weight on each of these shells,
we obtain
\begin{align}
 V_{7,1}-\frac{25}{4}
 &=14-\frac{451}{1051}-12\frac47-\frac{25}{4}
 =\frac{13647}{29428}>0,\label{eq:d7-shell1}\\
 V_{7,2}-\frac{25}{8}
 &=14-2\frac74-2\frac{143}{311}-10\frac{16}{25}-\frac{25}{8}
 =\frac{689}{12440}>0,\label{eq:d7-shell2}\\
 V_{7,3}-\frac{25}{12}
 &=14-3\frac{25}{16}-3\frac{41}{81}-8\frac{683}{963}-\frac{25}{12}
 =\frac{1697}{46224}>0.\label{eq:d7-shell3}
\end{align}
The first calculation retains the origin-edge contribution in the diagonal
$14$, as required by Proposition~\ref{prop:punctured}.
By the classification above, these calculations cover the remaining
nonzero vertices.

\section{Dimension eight}\label{sec:d8}

In dimension eight we use a different rational reciprocal field. Its
vertex weight has the required lower bound for $|x|^2\ge13$. We then
modify the field inside
\begin{equation}\label{eq:d8-core}
 \mathcal S=S_{12}:=\{x\in\Z^8:1\le |x|^2\le12\}.
\end{equation}
The correction is given by $26$ positive integers.

\subsection{The reciprocal field}

For $S=n+m$ and $\delta=m-n$, put
\begin{align}
 P(S,\delta)
 & =768S^4-3981S^3+5808S^2-5856S-480\notag\\
 &\qquad +(2000S^2+5184S-3648)\delta^2-3840\delta^4,
 \label{eq:d8-P}\\
 Q(S,\delta)
 & =\delta(1152S^3-4902S^2+2976S+4416+11520\delta^2),
 \label{eq:d8-Q}
\end{align}
and define
\begin{equation}\label{eq:d8-rational-field}
 \widehat a_8(n,m)=\frac{P(S,\delta)-Q(S,\delta)}
                         {P(S,\delta)+Q(S,\delta)}.
\end{equation}
We use this formula on edges with at least one endpoint outside
$\mathcal S$.

\begin{lemma}\label{lem:d8-factor-positive}
If $S\ge16$ and $\delta^2\le2S-1$, then $P(S,\delta)>|Q(S,\delta)|$.
Consequently, the factors in \eqref{eq:d8-rational-field} are positive
and reciprocal on every edge where they are used.
\end{lemma}

\begin{proof}
The inequality $\delta^4\le(2S-1)\delta^2$ gives
\begin{align*}
 \frac{P}{768}
 &\ge S^4-\frac{1327}{256}S^3+\frac{121}{16}S^2
       -\frac{61}{8}S-\frac58
       +\left(\frac{125}{48}S^2-\frac{13}{4}S+\frac14\right)\delta^2\\
 &\ge S^3(S-6).
\end{align*}
For the second inequality, the coefficient of $\delta^2$ is positive
for $S\ge2$, and the remaining difference is
\[
 \frac{209}{256}S^3+\frac{121S^2-122S-10}{16}>0.
\]
Indeed, the numerators of these two quadratic expressions, expanded
at $S=2$, are
\[
 125(S-2)^2+344(S-2)+200,
 \qquad
 121(S-2)^2+362(S-2)+230.
\]
Also
\[
 \frac{Q}{768}
 =\delta\left(\frac32S^3-\frac{817}{128}S^2
                 +\frac{31}{8}S+\frac{23}{4}+15\delta^2\right).
\]
The expression in parentheses is positive for $S\ge16$, since
$\frac32S-\frac{817}{128}>0$. It is at most $\frac32S^3$:
using $\delta^2\le2S-1$, the required inequality reduces to
\[
 \frac{817}{128}S^2-\frac{271}{8}S+\frac{37}{4}>0.
\]
After multiplication by $128$ and expansion at $S=16$, its left-hand
side has numerator
$817(S-16)^2+21808(S-16)+140960$.
Thus
\[
 \frac{|Q|}{768}\le\frac32\sqrt{2S-1}\,S^3.
\]
Since $S-6>0$ and
\[
 (S-6)^2-\frac94(2S-1)
 =(S-16)^2+\frac{31}{2}(S-16)+\frac{121}{4}>0,
\]
we obtain $P>|Q|$.

For a lattice edge $y=x\pm e_j$,
\[
 \delta^2-(2S-1)=4(x_j^2-|x|^2)\le0.
\]
If an endpoint has squared norm $n\ge13$, then the triangle
inequality gives
\[
 S\ge n+(\sqrt n-1)^2\ge27-2\sqrt{13}>19.
\]
As $S$ is an integer, $S\ge20$. The positivity just proved therefore
applies. Finally, $P(S,-\delta)=P(S,\delta)$ and
$Q(S,-\delta)=-Q(S,\delta)$, so interchanging the endpoints replaces
$\widehat a_8(n,m)$ by its reciprocal.
\end{proof}

To apply Lemma~\ref{lem:gluing}, define a field on all edges with
nonzero endpoints by
\begin{equation}\label{eq:d8-base-field}
 a^0_{xy}=
 \begin{cases}
  1,&x,y\in\mathcal S,\\
  \widehat a_8(|x|^2,|y|^2),&\text{otherwise}.
 \end{cases}
\end{equation}
This field is positive and reciprocal. Its values on edges inside
$\mathcal S$ will be replaced by ratios of a positive function.
Write
\begin{equation}\label{eq:d8-base-weight}
 V_8^0(x)=16-\sum_{\substack{y\sim x\\y\ne0}}a^0_{xy}.
\end{equation}

\subsection{The exterior estimate}

\begin{proposition}\label{prop:d8-exterior}
For every $x\in\Z^8$ with $|x|^2\ge13$,
\begin{equation}\label{eq:d8-exterior-sharp}
 V_8^0(x)>\frac9{|x|^2}.
\end{equation}
\end{proposition}

For real $n\ge13$ and $0\le s\le n$, define
\[
 f_n(s)=2-\widehat a_8(n,n+1+2\sqrt{s})
         -\widehat a_8(n,n+1-2\sqrt{s}).
\]
These evaluations are well defined. For either choice of sign,
\[
 S=2n+1\pm2\sqrt{s}\ge2n+1-2\sqrt n>16,
 \qquad
 \delta^2-(2S-1)=4(s-n)\le0,
\]
so Lemma~\ref{lem:d8-factor-positive} applies. Pairing opposite
neighbors gives
\begin{equation}\label{eq:d8-pairing}
 V_8^0(x)=\sum_{j=1}^8 f_n(x_j^2),\qquad n=|x|^2\ge13.
\end{equation}
There is no neighbor at the origin in this range. If $x_j<0$, the two
arguments in the definition of $f_n(x_j^2)$ are interchanged.

We will prove
\begin{equation}\label{eq:d8-minorant}
 f_n(s)\ge f_n(0)+\frac{s}{n}\left(\frac9n-8f_n(0)\right),
\end{equation}
strictly for $s>0$, for all real $n\ge17$ and $0\le s\le n$.
For $n=13,14,15,16$, we prove this inequality only when $s$ is the
square of an integer and $s\le n$.

Put
\begin{align}
 E(n,s)&=12288n^4+1944n^3-3892n^2+5402n+11117\notag\\
       &\qquad+4(40256n^2+30042n+29079)s+1280s^2,\notag\\
 O(n,s)&=2(33792n^3+13132n^2+26100n+32615)
          +56(3008n+1163)s,\notag\\
 U(n,s)&=6(1536n^3-964n^2-1124n+2527)\notag\\
       &\qquad+24(2304n^2+188n+4957)s+18432s^2,\notag\\
 T(n,s)&=12(1536n^3+1340n^2-2088n+5805)
          +48(1152n+1871)s.\label{eq:d8-paired-polynomials}
\end{align}
Substituting $S=2n+1+2z$ and $\delta=1+2z$ in
\eqref{eq:d8-P}--\eqref{eq:d8-Q} gives
\begin{align*}
 (P+Q)(2n+1+2z,1+2z)&=E(n,z^2)+zO(n,z^2),\\
 Q(2n+1+2z,1+2z)&=U(n,z^2)+zT(n,z^2).
\end{align*}
Set
\[
 D=E^2-sO^2,\qquad e=E(n,0),\qquad u_0=U(n,0).
\]
Here $D$ is the product of the two positive denominators at
$z=\sqrt{s}$ and $z=-\sqrt{s}$, and $e>0$. Since $1-\widehat a_8=2Q/(P+Q)$,
adding the two fractions gives
\[
 f_n(s)=\frac{4(UE-sTO)}{D},\qquad f_n(0)=\frac{4u_0}{e}.
\]
Define the polynomial $R(n,s)$ by
\begin{equation}\label{eq:d8-R}
 3sR(n,s)=4n^2e(UE-sTO)-4n^2u_0D-s(9e-32nu_0)D.
\end{equation}
The right-hand side vanishes at $s=0$; expansion gives a polynomial
of degree four in $s$ after division by $3s$. This definition also
specifies $R(n,0)$. Subtracting the affine expression yields
\begin{equation}\label{eq:d8-minorant-gap}
 f_n(s)-f_n(0)-\frac{s}{n}\left(\frac9n-8f_n(0)\right)
 =\frac{3sR(n,s)}{n^2eD}.
\end{equation}

\begin{lemma}\label{lem:d8-quartic}
For real $n\ge17$ and $0\le s\le n$, we have $R(n,s)>0$.
\end{lemma}

\begin{proof}
Write
\[
 p_n(v)=R(n,nv)=\sum_{j=0}^4a_j(n)v^j,
\]
and put
\[
 \theta=\frac27,\qquad B=p_n(\theta),\qquad L=p_n'(\theta),
 \qquad K=a_2+\frac{11}{7}a_3+\frac{12}{49}a_4.
\]
The polynomials in Appendix~\ref{app:d8-polynomials} give
\begin{equation}
 \begin{gathered}
 a_3=-64n^3U_0(n),\qquad a_4=1638400n^4Z_0(n),\\
 K=\frac{16n^2}{49}K_0(n),\qquad
 B=\frac{B_0(n)}{2401},\qquad
 L=-\frac{4n}{343}L_0(n).
 \end{gathered}
 \label{eq:d8-quartic-identities}
\end{equation}
Let
\[
 k_0=285604168335360,\quad b_0=40604352380928,
 \quad \ell_0=233062545752064,
\]
and set $k=K_0/k_0$, $b=B_0/b_0$, and $\ell=L_0/\ell_0$.
We first prove
\begin{equation}\label{eq:d8-three-bounds}
 k(n)\ge n^6(n^2-17n+12),\qquad
 b(n)\ge\frac12n^{11},\qquad
 0<\ell(n)\le n^8(n-14).
\end{equation}
Comparing the coefficients of the polynomials in the appendix gives
\begin{align*}
 k(n)&\ge n^8-17n^7+12n^6+13n^5-25n^4
             +5n^3-13n^2+10n-2,\\
 b(n)&\ge n^{11}+13n^{10}-241n^9-1815n^8
             -1130n^6-5n^5\\
     &\qquad-938n^4-821n^3-1871n^2-244,\\
 \ell(n)&\ge n^9-19n^8+72n^7-27n^6-82n^4-65n^2-21,\\
 \ell(n)&\le n^9-\frac{37}{2}n^8+73n^7+2n^5+158n^3+90n.
\end{align*}
For $k$, the terms after $12n^6$ have sum
\[
 n^4(13n-25)+n^2(5n-13)+10n-2>0.
\]
For $b$, the last six negative terms have absolute sum at most $4n^8$,
because
\[
 1130+\frac5n+\frac{938}{n^2}+\frac{821}{n^3}
       +\frac{1871}{n^4}+\frac{244}{n^6}
 \le1138<4\cdot17^2\le4n^2.
\]
It follows that $b(n)\ge n^8(n^3+13n^2-241n-1819)$.
With $t=n-17\ge0$, we have
\[
 2\left(\frac12n^3+13n^2-241n-1819\right)
 =t^3+77t^2+1269t+595>0,
\]
which proves the asserted lower bound for $b$.
For the lower bound for $\ell$, use
\[
 \frac{27}{n}+\frac{82}{n^3}+\frac{65}{n^5}+\frac{21}{n^7}<5,
 \qquad
 n^2-19n+67=t^2+15t+33>0.
\]
The displayed lower bound for $\ell(n)$ is therefore greater than
$n^7(n^2-19n+67)>0$. For the upper bound, subtraction and division
by $n^7$ give
\[
 \frac{\ell(n)-n^8(n-14)}{n^7}
 \le-\frac92n+73+\frac2{n^2}+\frac{158}{n^4}+\frac{90}{n^6}
 <-\frac92\cdot17+73+3<0.
\]
This proves \eqref{eq:d8-three-bounds}.

The same coefficient comparisons give
\begin{align*}
 \frac{U_0(n)}{26837967175680}
 &\ge n^6-n^5-2n^4-n-1
   =(n+1)\bigl(n^4(n-2)-1\bigr)>0,\\
 \frac{Z_0(n)}{61440}
 &\ge n^4-2n^3-n^2-1
   =n^2(n^2-2n-1)-1>0.
\end{align*}
Thus $a_3<0$, $a_4>0$, and $K>0$; for the last assertion use
$n^2-17n+12=t^2+17t+12>0$.

Put
\[
 \rho=\frac{\ell_0^2}{4k_0b_0}
      =\frac{47628756003}{40674421760}<\frac65.
\]
By \eqref{eq:d8-quartic-identities},
\[
 4KB-L^2=\frac{64k_0b_0n^2}{117649}
             \bigl(k(n)b(n)-\rho\ell(n)^2\bigr).
\]
All lower bounds used in the product are positive. Hence
\eqref{eq:d8-three-bounds} implies
\begin{align*}
 k(n)b(n)-\rho\ell(n)^2
 &\ge\frac12n^{17}(n^2-17n+12)-\frac65n^{16}(n-14)^2\\
 &=\frac{n^{16}}{10}(5n^3-97n^2+396n-2352)\\
 &=\frac{n^{16}}{10}(5t^3+158t^2+1433t+912)>0.
\end{align*}
In particular, $4KB-L^2>0$.

The exact Taylor identity for the quartic is
\[
 p_n(v)=B+L(v-\theta)
 +\bigl[a_2+a_3(v+2\theta)
         +a_4(v^2+2\theta v+3\theta^2)\bigr](v-\theta)^2.
\]
For $0\le v\le1$, the expression in brackets minus $K$ is
\[
 a_3(v-1)+a_4v\left(v+\frac47\right)\ge0,
\]
because $a_3<0$ and $a_4>0$. Therefore
\[
 p_n(v)\ge
 K\left(v-\theta+\frac{L}{2K}\right)^2
 +\frac{4KB-L^2}{4K}>0.
\]
Taking $v=s/n$ proves the lemma.
\end{proof}

\begin{proof}[Proof of Proposition~\ref{prop:d8-exterior}]
For $n\ge17$, Lemma~\ref{lem:d8-quartic} and
\eqref{eq:d8-minorant-gap} prove \eqref{eq:d8-minorant}, with strict
inequality for $s>0$.
For $n=13,14,15,16$, an integer coordinate has square
$s\in\{0,1,4,9\}$, with $s=16$ also possible when $n=16$.
The case $s=0$ is equality in \eqref{eq:d8-minorant}; all other
values of $R(n,s)$ are positive by Table~\ref{tab:d8-small-radii}.
Its entries are obtained by substitution in \eqref{eq:d8-R}.
Thus \eqref{eq:d8-minorant} holds at each $s=x_j^2$ for every
lattice point with $n=|x|^2\ge13$. Summing gives
\[
 V_8^0(x)\ge8f_n(0)+\frac{\sum_jx_j^2}{n}
                   \left(\frac9n-8f_n(0)\right)=\frac9n.
\]
At least one $x_j^2$ is positive, since their sum is $n>0$.
That summand has strict inequality, proving
\eqref{eq:d8-exterior-sharp}.
\end{proof}

\begin{table}[htbp]
\centering
\caption{The remaining exterior substitutions.}
\label{tab:d8-small-radii}
\small
\begin{tabular}{ccr}
\toprule
$n$&$j$&$R(n,j^2)$\\
\midrule
13 & 1 & 583849712261030131522731\\
13 & 2 & 733024632574579780479\\
13 & 3 & 441619762576565385246219\\
14 & 1 & 1317917070542853833755017\\
14 & 2 & 9174265721464800028485\\
14 & 3 & 1114938244462327757658025\\
15 & 1 & 2791678628394353680044615\\
15 & 2 & 67238931820085655573195\\
15 & 3 & 2226888958547438103814695\\
16 & 1 & 5602880867709276856820901\\
16 & 2 & 268980424296033976089105\\
16 & 3 & 3923039680741798133135685\\
16 & 4 & 36581949289283051630721\\
\bottomrule
\end{tabular}
\end{table}

\subsection{The correction near the origin}

We specify a positive function $h$ on $\mathcal S$, constant under
coordinate permutations and sign changes. As before, $1^k$ denotes
$k$ coordinates equal to $1$, and trailing zeros are omitted.
The representatives are
\[
 \begin{gathered}
 (1^k)\quad(1\le k\le8),\qquad
 (2,1^k)\quad(0\le k\le7),\\
 (2^2,1^k)\quad(0\le k\le4),\qquad
 (2^3),\qquad
 (3,1^k)\quad(0\le k\le3).
 \end{gathered}
\]
These $26$ classes exhaust $\mathcal S$. A coordinate of absolute
value at least $4$ is impossible. If a coordinate has absolute value
$3$, no other coordinate can have absolute value $2$ or $3$, and
at most three coordinates can have absolute value $1$. Otherwise,
there are at most three coordinates of absolute value $2$.
With zero, one, two, or three such coordinates, the norm and dimension
bounds give exactly the first four families above.

For $x\in\mathcal S$, put $n=|x|^2$ and define
\begin{align}
 I_x&=\sum_{\substack{y\sim x\\y\in\mathcal S}}h(y),\notag\\
 D_x&=16-\frac9n
       -\sum_{\substack{y\sim x\\y\in X_8\setminus\mathcal S}}
                  \widehat a_8(n,|y|^2).
 \label{eq:d8-local-diagonal}
\end{align}
The diagonal remains $16$ at a vertex adjacent to the origin,
as in \eqref{eq:punctured-weight}.

The possible squared norms $(n,m)$ on an edge from $\mathcal S$
to its complement are listed in Table~\ref{tab:d8-boundary}.
For each pair the table gives integers $A_{nm}$ and
\begin{equation}\label{eq:d8-boundary-gap}
 G_{nm}=A_{nm}(P+Q)(n+m,m-n)
              -10^7(P-Q)(n+m,m-n)>0.
\end{equation}
Because $P+Q>0$, this proves
\begin{equation}\label{eq:d8-boundary-bound}
 \widehat a_8(n,m)<\frac{A_{nm}}{10^7}.
\end{equation}
Set
\[
 C_x=\sum_{\substack{y\sim x\\y\in X_8\setminus\mathcal S}}
                     A_{n,|y|^2}.
\]

\begin{table}[htbp]
\centering
\caption{Bounds for the crossing edges, with $G_{nm}$ as in
\eqref{eq:d8-boundary-gap}.}
\label{tab:d8-boundary}
\small
\begin{tabular}{ccrr}
\toprule
$n$&$m$&$A_{nm}$&$G_{nm}$\\
\midrule
8 & 13 & 5118020 & 21994980\\
9 & 14 & 5383828 & 113119500\\
9 & 16 & 4661126 & 196991750\\
10 & 13 & 6754838 & 34551138\\
10 & 15 & 5625123 & 327802375\\
10 & 17 & 4889148 & 312827076\\
11 & 14 & 6963180 & 20419980\\
11 & 16 & 5845030 & 424342050\\
11 & 18 & 5099458 & 555219786\\
12 & 13 & 8833696 & 29639328\\
12 & 15 & 7147697 & 339691911\\
12 & 17 & 6046017 & 288001201\\
12 & 19 & 5294573 & 693803591\\
\bottomrule
\end{tabular}
\end{table}

Here is the neighbor count used in both tables. If $t_r$ coordinates
have absolute value $r$, then the zero coordinates give $2t_0$
neighbors of squared norm $n+1$. For each $r\ge1$, there are $t_r$
neighbors of squared norm $n+1-2r$ and $t_r$ of squared norm
$n+1+2r$. Each is obtained by changing one coordinate of absolute
value $r$ to $r-1$ or $r+1$, respectively. This specifies its
symmetry class as well as its squared norm. The origin is omitted
from $I_x$. Only the increasing changes can leave $\mathcal S$.
In particular,
\begin{equation}\label{eq:d8-crossing-count}
 C_x=2t_0\boldsymbol{1}_{\{n+1>12\}}A_{n,n+1}
       +\sum_{r=1}^3t_r\boldsymbol{1}_{\{n+1+2r>12\}}
                            A_{n,n+1+2r},
\end{equation}
where terms with zero multiplicity or false conditions are omitted.
This rule gives exactly the thirteen pairs in
Table~\ref{tab:d8-boundary}.

Table~\ref{tab:d8-local} gives $h(x)$, the resulting values of $I_x$
and $C_x$, and integers $d_x,r_x$ satisfying
\begin{equation}\label{eq:d8-integer-local}
 d_xh(x)-10^5I_x=r_x>0,
 \qquad
 100n d_x+nC_x\le160000000n-90000000.
\end{equation}
The second inequality is equivalent to
\[
 \frac{d_x}{10^5}\le16-\frac9n-\frac{C_x}{10^7}\le D_x,
\]
where the last inequality uses \eqref{eq:d8-boundary-bound}.
All the quantities in \eqref{eq:d8-integer-local} are integers
given in the table. For example, at $x=(1)$ there are fourteen
neighbors of type $(1^2)$ and one of type $(2)$, giving
\[
 I_x=14\cdot976301+465840=14134054,
 \qquad
 7\cdot2019151-14134054=3.
\]
At $x=(2,1^4)$ the only crossing edge has $(n,m)=(8,13)$,
so $C_x=5118020$, and
\[
 1436319\cdot132937-10^5\cdot1909398=138903>0.
\]
The neighbor rule and \eqref{eq:d8-integer-local} verify the other
rows in the same way. Consequently,
\begin{equation}\label{eq:d8-local-positive}
 D_xh(x)-I_x\ge\frac{r_x}{10^5}>0,
 \qquad x\in\mathcal S.
\end{equation}
Signed coordinate permutations preserve adjacency, squared norms,
and the values of $h$. They therefore preserve $I_x$, $C_x$, and
$D_x$, so the inequalities for the representatives prove
\eqref{eq:d8-local-positive} throughout $\mathcal S$.

\begin{table}[htbp]
\centering
\caption{The local correction. Neighbor sums include multiplicities.}
\label{tab:d8-local}
\small
\setlength{\tabcolsep}{4pt}
\begin{tabular}{lrrrrr}
\toprule
$x$&$h(x)$&$I_x$&$C_x$&$d_x$&$r_x$\\
\midrule
$(1)$ & 2019151 & 14134054 & 0 & 700000 & 300000\\
$(1^{2})$ & 976301 & 11227456 & 0 & 1150000 & 550000\\
$(1^{3})$ & 548569 & 7131391 & 0 & 1300000 & 600000\\
$(1^{4})$ & 355162 & 4883472 & 0 & 1375000 & 550000\\
$(1^{5})$ & 253086 & 3593809 & 0 & 1420000 & 1220000\\
$(1^{6})$ & 192219 & 2787172 & 0 & 1450000 & 350000\\
$(1^{7})$ & 152422 & 2242769 & 0 & 1471428 & 1098616\\
$(1^{8})$ & 124707 & 1855008 & 0 & 1487500 & 862500\\
$(2)$ & 465840 & 6405288 & 0 & 1375000 & 1200000\\
$(2,1)$ & 303163 & 4304912 & 0 & 1420000 & 260000\\
$(2,1^{2})$ & 216956 & 3145856 & 0 & 1450000 & 600000\\
$(2,1^{3})$ & 166127 & 2444425 & 0 & 1471428 & 1419356\\
$(2,1^{4})$ & 132937 & 1909398 & 5118020 & 1436319 & 138903\\
$(2,1^{5})$ & 109828 & 1588283 & 5383828 & 1446161 & 670308\\
$(2,1^{6})$ & 92546 & 970298 & 46154151 & 1048458 & 794068\\
$(2,1^{7})$ & 79454 & 772529 & 54587290 & 972308 & 859832\\
$(2^{2})$ & 145309 & 2012726 & 10236040 & 1385139 & 562951\\
$(2^{2},1)$ & 117200 & 1631796 & 10767656 & 1392323 & 655600\\
$(2^{2},1^{2})$ & 97618 & 1232326 & 24759922 & 1262400 & 363200\\
$(2^{2},1^{3})$ & 83209 & 992162 & 32579600 & 1192385 & 963465\\
$(2^{2},1^{4})$ & 72239 & 552492 & 76017606 & 764823 & 848697\\
$(2^{3})$ & 76395 & 351600 & 106475011 & 460249 & 722355\\
$(3)$ & 141855 & 2061700 & 4661126 & 1453388 & 354740\\
$(3,1)$ & 113990 & 1588510 & 11643986 & 1393560 & 904400\\
$(3,1^{2})$ & 95291 & 1265386 & 19025818 & 1327923 & 510593\\
$(3,1^{3})$ & 82045 & 452000 & 97407232 & 550927 & 805715\\
\bottomrule
\end{tabular}
\end{table}

\begin{proposition}[Sharp lower bound in dimension eight]\label{prop:d8-lower}
For every $u\in\Dzero(\Z^8)$,
\begin{equation}\label{eq:d8-sharp-lower}
 \E_8(u)\ge9\HH_8(u),
\end{equation}
and the inequality is strict for every nonzero finitely supported $u$.
\end{proposition}

\begin{proof}
Use Lemma~\ref{lem:gluing} with $S=\mathcal S$, $\kappa=9$, the
base field \eqref{eq:d8-base-field}, and the function $h$ in
Table~\ref{tab:d8-local}. Explicitly, put
\[
 a_{xy}=
 \begin{cases}
  h(y)/h(x),&x,y\in\mathcal S,\\
  \widehat a_8(|x|^2,|y|^2),&\text{otherwise},
 \end{cases}
 \qquad x,y\in X_8,\quad x\sim y.
\]
The field is positive and reciprocal. Every crossing and exterior
edge is unchanged, so Proposition~\ref{prop:d8-exterior} gives
$V(x)>9/|x|^2$ for $x\in X_8\setminus\mathcal S$. For
$x\in\mathcal S$, \eqref{eq:d8-local-positive} gives
\[
 V(x)-\frac9{|x|^2}
 =\frac{D_xh(x)-I_x}{h(x)}
 \ge\frac{r_x}{10^5h(x)}>0.
\]
Thus
\begin{equation}\label{eq:d8-global-weight}
 V(x)>\frac9{|x|^2},\qquad x\ne0.
\end{equation}
Applying \eqref{eq:punctured} and dropping the nonnegative
edge-square sum proves \eqref{eq:d8-sharp-lower}. If $u\ne0$ and
$u(0)=0$, at least one nonzero vertex has $|u(x)|^2>0$.
Therefore \eqref{eq:d8-global-weight} also proves strictness.
\end{proof}

The dimension-eight field and local correction were found with
computational assistance. Their verification uses only the polynomial
identities and integer inequalities displayed above and in
Appendix~\ref{app:d8-polynomials}.

\section{Proof of the continuum upper bound}\label{sec:sharpness}

\begin{proof}[Proof of Proposition~\ref{prop:upper}]
Fix $d\ge3$, put $\alpha=(d-2)/2$, and choose a nonzero real-valued
$\phi\in C_c^\infty((0,1))$. For $L>1$, define
\begin{equation}\label{eq:phi-L}
 \varphi_L(x)=|x|^{-\alpha}
              \phi\left(\frac{\log|x|}{L}\right),\qquad x\ne0,
\end{equation}
and set $\varphi_L(0)=0$. The function is supported in an annulus
separated from the origin and vanishes near both boundaries of that
annulus. Its extension is therefore in $C_c^\infty(\R^d)$.

For a radial function $\varphi(r)=r^{-\alpha}g(\log r)$, polar
coordinates and $2\alpha=d-2$ give
\begin{align}
 \int_{\R^d}\frac{|\varphi(x)|^2}{|x|^2}\dd x
 &=\omega_{d-1}\int_\R|g(t)|^2\dd t,\label{eq:radial-mass}\\
 \int_{\R^d}|\nabla\varphi(x)|^2\dd x
 &=\omega_{d-1}\int_\R|g'(t)-\alpha g(t)|^2\dd t,
 \label{eq:radial-energy}
\end{align}
where $t=\log r$ and $\omega_{d-1}=|\mathbb S^{d-1}|$.
For \eqref{eq:phi-L}, $g(t)=\phi(t/L)$. The mixed term vanishes because
\[
 \operatorname{Re}\int_\R g'(t)\overline{g(t)}\dd t
 =\frac12\int_\R (|g(t)|^2)'\dd t=0.
\]
Changing variables from $t$ to $t/L$ yields
\begin{align}
 \int_{\R^d}\frac{|\varphi_L|^2}{|x|^2}\dd x
 &=\omega_{d-1}L\int_0^1|\phi|^2,\label{eq:continuous-mass}\\
 \int_{\R^d}|\nabla\varphi_L|^2\dd x
 &=\omega_{d-1}L\left(C_d\int_0^1|\phi|^2
               +\frac1{L^2}\int_0^1|\phi'|^2\right).
 \label{eq:continuous-energy}
\end{align}
Consequently,
\begin{equation}\label{eq:continuous-quotient}
 \frac{\int_{\R^d}|\nabla\varphi_L|^2\dd x}
      {\int_{\R^d}|\varphi_L|^2/|x|^2\dd x}
 =C_d+\frac1{L^2}
       \frac{\int_0^1|\phi'|^2}{\int_0^1|\phi|^2}.
\end{equation}

We now pass to the lattice. For fixed $L$ and $N\in\mathbb N$, set
\begin{equation}\label{eq:u-LN}
 u_{L,N}(x)=\varphi_L(x/N),\qquad x\in\Z^d.
\end{equation}
This is finitely supported and vanishes at the origin.
The difference quotients
\[
 N\bigl(\varphi_L(y+e_j/N)-\varphi_L(y)\bigr)
 =\int_0^1\partial_j\varphi_L(y+t e_j/N)\dd t
\]
converge uniformly to $\partial_j\varphi_L(y)$.
Their supports lie in one compact set depending on $L$ but not on $N$.
Writing the scaled energy as a Riemann sum therefore gives
\begin{align}
 N^{2-d}\E_d(u_{L,N})
 &=\frac1{N^d}\sum_{x\in\Z^d}\sum_{j=1}^d
 \left|N\bigl(\varphi_L((x+e_j)/N)-\varphi_L(x/N)\bigr)\right|^2
 \notag\\
 &\longrightarrow \int_{\R^d}|\nabla\varphi_L(y)|^2\dd y.
 \label{eq:energy-riemann}
\end{align}
Indeed, the difference between this sum and the Riemann sum of
$\sum_j|\partial_j\varphi_L|^2$ is bounded by the uniform error
times $N^{-d}$ times the number of grid points in a fixed compact set;
the last factor is bounded uniformly in $N$.
Likewise,
\begin{align}
 N^{2-d}\HH_d(u_{L,N})
 &=\frac1{N^d}\sum_{x\ne0}
       \frac{|\varphi_L(x/N)|^2}{|x/N|^2}\longrightarrow\int_{\R^d}\frac{|\varphi_L(y)|^2}{|y|^2}\dd y.
 \label{eq:mass-riemann}
\end{align}
The mass integrand extends continuously by zero near the origin and is
compactly supported. Its integral is strictly positive, so the sampled
function is nonzero for sufficiently large $N$, and its quotient
converges to \eqref{eq:continuous-quotient}.

For every integer $L\ge2$, choose $N(L)$ large enough that $u_{L,N(L)}$
is nonzero and its quotient differs from \eqref{eq:continuous-quotient}
by at most $L^{-2}$. With $u_L=u_{L,N(L)}$, the quotient tends to $C_d$.
This proves Proposition~\ref{prop:upper}.
\end{proof}

\section{Proof of the main results}\label{sec:completion}

For $d=5$, use the unmodified base field and write $W_5=W_5^0$.  For $d=6$,
use the field in \eqref{eq:d6-modification} and its full-lattice weight $W_6$.
For $d=7$, let $V_7$ be the punctured weight constructed in
\eqref{eq:d7-modification}.  For $d=8$, let $V_8$ be the punctured weight
constructed in Proposition~\ref{prop:d8-lower}.  Put
\[
 \kappa_d=C_d,
 \qquad d=5,6,7,8,
\]
and for $x\ne0$ define
\begin{equation}\label{eq:epsilon}
 \varepsilon_d(x)=
 \begin{cases}
 W_d(x)-C_d/|x|^2,&d=5,6,\\
 V_d(x)-C_d/|x|^2,&d=7,8.
 \end{cases}
\end{equation}
The exterior estimates, the local calculations in dimensions five through
seven, and Proposition~\ref{prop:d8-lower} give
\[
 \varepsilon_d(x)>0,
 \qquad x\ne0,
\]
in all four cases.  For the two full-lattice cases,
\[
 \gamma_5=\frac{720}{161},\qquad \gamma_6=4,
\]
and set $\gamma_7=\gamma_8=0$.

\begin{corollary}[Spatial remainders]\label{cor:improvement}
For every finitely supported $u:\Z^d\to\C$ when $d=5,6$, and for every
$u\in\Dzero(\Z^d)$ when $d=7,8$,
\begin{equation}\label{eq:general-remainder}
 \E_d(u)-C_d\HH_d(u)
 \ge \gamma_d|u(0)|^2
       +\sum_{x\ne0}\varepsilon_d(x)|u(x)|^2.
\end{equation}
The right-hand side is strictly positive for every nonzero admissible $u$.
\end{corollary}

\begin{proof}
For $d=5,6$, subtract $C_d\HH_d(u)$ from the full-lattice identity
\eqref{eq:edge}.  For $d=7,8$, use the punctured identity
\eqref{eq:punctured}.  In each case, the remaining edge-square sum is
nonnegative and the pointwise coefficients above are strictly positive.
\end{proof}

\begin{proof}[Proof of Theorem~\ref{thm:main}]
For $d=5,6,7$, Corollary~\ref{cor:improvement} gives $\lambda_d\ge C_d$.
For $d=8$, Proposition~\ref{prop:d8-lower} gives $\lambda_8\ge9=C_8$.
Proposition~\ref{prop:upper} gives the reverse inequality
$\lambda_d\le C_d$ in every dimension.  Hence
\[
 \lambda_d=C_d,
 \qquad d=5,6,7,8.
\]
The same minimizing sequence proves sharpness in the enlarged full-lattice
classes for $d=5,6$.  The strict remainders exclude nonzero finitely supported
extremizers in all stated admissible classes.
\end{proof}

\begin{remark}[Alternative proof in dimension four]\label{rem:d4-check}
Set $d=4$ in the quadratic field \eqref{eq:base-field}. 
For $r=1/2$, the coefficients $\mathsf A,\mathsf B,\mathsf G$
in \eqref{eq:general-coefficients} satisfy
$\mathsf A(n),\mathsf G(n)>0$ for $n\ge2$ and
$\mathsf B(n)>0$ for $n\ge8$.
Indeed, the expansions of $\mathsf A$ and $\mathsf G$
in powers of $n-2$, and of $\mathsf B$ in powers of $n-8$,
have strictly positive coefficients.
For the remaining integer values, direct substitution gives
\[
 4\mathsf A(n)\mathsf G(n)-\mathsf B(n)^2
 \ge \frac{90737675}{4}>0,
 \qquad n=2,\ldots,7.
\]
Completing the square for these six values, and using
coefficient positivity for $n\ge8$, therefore gives
\[
 \mathsf A(n)s^2+\mathsf B(n)s+\mathsf G(n)>0
 \qquad\text{for integers }n\ge2,\quad s\ge0.
\]
The same
affine-minorant computation gives
\[
 W_4^0(x)>\frac1{|x|^2},
 \qquad |x|^2\ge2,
\]
while the first-shell calculation in Section~\ref{sec:local} gives
\[
 W_4^0(0)=\frac{32}{11},
 \qquad
 W_4^0(e_1)-1=\frac{499}{1645}>0.
\]
The full-lattice decomposition therefore yields $\E_4(u)\ge\HH_4(u)$, with
strict inequality for nonzero finitely supported $u$.  Proposition~\ref{prop:upper}
gives the reverse bound for the sharp coefficient, so $\lambda_4=1$.  This
recovers the result of \cite{Lizama2026} by a different method.
\end{remark}

\begin{remark}[Further problems]
It remains desirable to find a simpler choice of the dimension-eight
reciprocal field and local correction, and to understand structurally
why the transition occurs between dimensions eight and nine.
Another open problem is to determine the exact value of $\lambda_9$.
\end{remark}

\appendix

\section{The scalar polynomial inequalities}\label{app:polynomials}

We give the identities used in Lemma~\ref{lem:quadratic-positivity}.
For $r=(d-2)/4$, use the coefficients in
\eqref{eq:general-coefficients} and set
\[
 A_d=\frac{\mathsf A}{\theta_d},\qquad
 B_d=\frac{\mathsf B}{\theta_d},\qquad
 G_d=\frac{\mathsf G}{\theta_d},
\]
\[
 P_d(n,s)=A_d(n)s^2+B_d(n)s+G_d(n).
\]
The parameters are
\[
\begin{array}{c|rrrr}
 d&L_d&\theta_d&c_d&\eta_d\\\hline
 5&2&7/147456&1/2&8\\
 6&3&16&1/2&4\\
 7&4&3/16384&2/5&8
\end{array}
\]
Set $t=n-L_d$ and $\ell_d(n)=n/4+c_d$. Then
\begin{equation}\label{eq:shifted-square}
 P_d(n,s)
 =A_d(n)(s-\ell_d(n))^2+\eta_d Q_d(t)s+T_d(t),
\end{equation}
where
\[
 Q_d(t)=\sum_{k=0}^4q_{d,k}t^k,\qquad
 T_d(t)=\sum_{k=0}^4\tau_{d,k}t^k.
\]
The coefficients are given in the following two tables, in ascending
powers of $t$.
\begin{center}
\begin{tabular}{crrrrr}
\toprule
$d$&$q_{d,0}$&$q_{d,1}$&$q_{d,2}$&$q_{d,3}$&$q_{d,4}$\\
\midrule
5&1953691&1920393&1555752&1243200&313344\\
6&49&52&35&14&2\\
7&91897&2534875&2627848&858816&88064\\
\bottomrule
\end{tabular}
\end{center}
\begin{center}
\begin{tabular}{crrrrr}
\toprule
$d$&$\tau_{d,0}$&$\tau_{d,1}$&$\tau_{d,2}$&$\tau_{d,3}$&$\tau_{d,4}$\\
\midrule
5&5738185&48059472&58320903&24299256&3329984\\
6&15&376&341&102&10\\
7&7231003&67377072&45139381&10287960&779584\\
\bottomrule
\end{tabular}
\end{center}

Substitution into \eqref{eq:general-coefficients} gives
\begin{align*}
 B_d(n)+2A_d(n)\ell_d(n)&=\eta_d Q_d(n-L_d),\\
 G_d(n)-A_d(n)\ell_d(n)^2&=T_d(n-L_d).
\end{align*}
Expanding the square in \eqref{eq:shifted-square} gives the coefficient
$A_d$ of $s^2$, the coefficient
$-2A_d\ell_d+\eta_dQ_d=B_d$ of $s$, and the constant term
$A_d\ell_d^2+T_d=G_d$.
This proves \eqref{eq:shifted-square}, and all coefficients in the two
tables are strictly positive.

\section{The dimension-eight polynomials}\label{app:d8-polynomials}

The polynomials used in \eqref{eq:d8-quartic-identities} are
\begin{align*}
 U_0(n)&=26837967175680n^{6}-8736198115328n^{5}-45366892868096n^{4}+39046933756664n^{3}\\
 &\qquad+30900487686908n^{2}-2306003601470n-2055196710591,\\[3pt]
 Z_0(n)&=61440n^{4}-67528n^{3}-60260n^{2}+145522n-33351,\\[3pt]
 K_0(n)&=285604168335360n^{8}-4790553012109312n^{7}+3502945568485376n^{6}+3934646221764768n^{5}\\
 &\qquad-7100287085715120n^{4}+1600828145517952n^{3}\\
 &\qquad-3654569877129056n^{2}+2970970276805622n-516945403230669,\\[3pt]
 B_0(n)&=40604352380928n^{11}+542739440861184n^{10}\\
 &\qquad-9775372273390080n^{9}-73658469568446208n^{8}\\
 &\qquad+110311315618614784n^{7}-45861436690164992n^{6}\\
 &\qquad-188316503762368n^{5}-38058456104795552n^{4}
 -33327827430984352n^{3}\\
  &\qquad-75953851901356240n^{2}+71743315540970206n-9896377011029439,\\[3pt]
 L_0(n)&=233062545752064n^{9}-4339817347145728n^{8}\\
 &\qquad+16983147052856704n^{7}-6077516035207744n^{6}\\
&\qquad+342599359470176n^{5}-19058245905170576n^{4}+36766047575745736n^{3}\\
&\qquad-14952238030883916n^{2}+20854347764951890n-4772466016346427.
\end{align*}
To verify \eqref{eq:d8-quartic-identities}, expand \eqref{eq:d8-R}
using \eqref{eq:d8-paired-polynomials}, and substitute $s=nv$.
The coefficients of $v^3$ and $v^4$ give $a_3$ and $a_4$.
The remaining identities follow by evaluating $p_n(2/7)$,
$p_n'(2/7)$, and $a_2+(11/7)a_3+(12/49)a_4$.
Thus every polynomial used in the proof is specified by the displayed
formulas.


\begin{thebibliography}{99}

\bibitem{Alpay2026}
N.~Alpay,
\emph{The sharp discrete Hardy inequality on $\mathbb{Z}^3$},
arXiv preprint
\href{https://arxiv.org/abs/2608.25262}{arXiv:2608.25262},
2026.


\bibitem{Gupta2023}
S.~Gupta,
\emph{Hardy and Rellich inequality on lattices},
Calc. Var. Partial Differential Equations \textbf{62} (2023),
Paper No.~81, 18~pp.
\href{https://doi.org/10.1007/s00526-022-02407-0}
{doi:10.1007/s00526-022-02407-0}.

\bibitem{Gupta2026}
S.~Gupta,
\emph{Sharp asymptotics for higher-order Hardy constants on lattices},
arXiv preprint
\href{https://arxiv.org/abs/2607.15181}{arXiv:2607.15181},
2026.

\bibitem{HuangYe2026}
X.~Huang and D.~Ye,
\emph{Asymptotically sharp Hardy--Rellich inequalities on lattices},
arXiv preprint
\href{https://arxiv.org/abs/2607.23057}{arXiv:2607.23057},
2026.


\bibitem{LizamaAsymptotics2026}
C.~Lizama,
\emph{Three-term asymptotics for discrete Hardy--Rellich constants
in high dimension},
arXiv preprint
\href{https://arxiv.org/abs/2608.30888}{arXiv:2608.30888},
2026.

\bibitem{Lizama2026}
C.~Lizama,
\emph{Sharp discrete Hardy constants in dimensions three and four
and strict upper bounds from dimension nine},
arXiv preprint
\href{https://arxiv.org/abs/2608.30831}{arXiv:2608.30831},
2026.


\end{thebibliography}
\end{document}